\documentclass[12pt,reqno]{amsart}

\usepackage{amsmath,amsfonts,amssymb,amscd,verbatim,delarray,verbatim}

\theoremstyle{plain}
\newtheorem{theorem}{Theorem}[section]

\newtheorem{Corollary}[theorem]{Corollary}

\newtheorem{Proposition}[theorem]{Proposition}
\newtheorem{Lemma}[theorem]{Lemma}

\theoremstyle{definition}

\newtheorem{Definition}[theorem]{Definition}

\newtheorem{Remark}[theorem]{Remark}
\newtheorem{Example}[theorem]{Example}

\newcommand{\thmref}[1]{Theorem~\ref{#1}}
\newcommand{\secref}[1]{Section~\ref{#1}}

\newcommand{\lemref}[1]{Lemma~\ref{#1}}
\newcommand{\corref}[1]{Corollary~\ref{#1}}

\newcommand{\remarkref}[1]{Remark~\ref{#1}}
\newcommand{\defnref}[1]{Definition~\ref{#1}}
\newcommand{\propref}[1]{Proposition~\ref{#1}}
\newcommand{\diagramref}[1]{Diagram~\ref{#1}}
\newcommand{\eqnref}[1]{Equation~\ref{#1}}
 \DeclareMathOperator{\SL}{SL}
\DeclareMathOperator{\PGL}{PGL} 
\DeclareMathOperator{\tr}{tr} 
\newcommand{\BZ}{\mathbb{Z}}\newcommand{\BC}{\mathbb{C}}
\newcommand{\BN}{\mathbb{N}}
\newcommand{\BP}{\mathbb{P}}\newcommand{\BR}{\mathbb{R}}
\newcommand{\BF}{\mathbb{F}}
\newcommand{\BA}{\mathbb{A}}\newcommand{\BQ}{\mathbb{Q}}

\newcommand{\al}{\alpha}
\newcommand{\be}{\beta}
\newcommand{\ga}{\gamma}

\newcommand{\ov}{\overline}

\newcommand{\CE}{\mathcal{E}}\newcommand{\CK}{\mathcal{K}}
\newcommand{\CF}{\mathcal{F}}

\newcommand{\CO}{\mathcal{O}}

\newtheorem{Theorem}[theorem]{Theorem}
\newtheorem{Diagram}[theorem]{Diagram}

\newcommand{\LND}{\rm {LND}}\newcommand{\Bir}{\rm {Bir}}
\newcommand{\pl}{\partial}\newcommand{\sg}{\rm {sgn}}
\newcommand{\Aut}{\rm {Aut}}\newcommand{\Tr}{\rm {Tr}}\newcommand{\ML}{\rm {ML}}
\begin{document}

\title[ ML-invariant and automorphism groups of   word varieties ] {  ML-invariant and automorphism groups of  certain word varieties in $\SL(2,\BC)^2.$}

\author{Tatiana Bandman }

\address{Bandman: Department of Mathematics, Bar-Ilan University, 52900 Ramat Gan, ISRAEL}
\email{bandman@math.biu.ac.il}

\dedicatory{Dedicated to Leonid Makar-Linanov for his 80th birthday}
\begin{abstract}

 For a fixed element $g\in\SL(2,\BC)$ and a word $w=[x^n,y^m]$ we consider the automorphism group $\Aut(S_{g})$ of the affine threefold 
$S_{g}=\{(x,y)\in \SL(2,\BC)^2 \ | w(x,y)=g\}.$  We prove that Makar-Limanov invariant $\ML(S_{g})=\CO(S_{g})$  and $\Aut(S_{g})$  is Jordan.\end{abstract}
\maketitle

\section{Introduction}\label{sec:intro}

 In this  paper  we study an affine three dimensional  subvarieties $S_g:=S_{w,g}$ of $\SL(2,\BC)\times\SL(2,\BC)$ defined by 
the equation  
\begin{equation}\label {Intro:1}     S_{w,g}=\{(x,y)\in\SL(2,\BC)\times\SL(2,\BC) \ |
w(x,y)=g\}. 
 \end{equation}
 Here $w$ is a word in two letters (i.e., an element of the free group $\CF_2$ on two generators) and    $g$ is an  element of the special linear group $\SL(2,\BC).$

 Word equations in $\SL(2,k)$   over   different  fields   $k$  have been the subject of extensive research in recent decades (see, e.g., the survey paper  \cite{BGK}).  Borel's dominance theorem \cite{Bo},  states that if $G$ is a
connected semisimple algebraic group, and $w\ne 1,$  then
the word map
$$w\colon G^2\to G, \ (x,y)\mapsto w(x,y)\in G,$$
 is dominant.
 It was proven in \cite{BZ} that  if $k$ is an algebraically closed field of zero characteristic  then for every word $w$ and every $g\in \SL(2,k)$ 
the set $S_{g}$ is non-empty  subset of $\SL(2,k)\times\SL(2,k)$  provided $\al:=tr(g)\ne \pm 2.$  
 The commutator word $w(x,y)=[x,y]=xyx^{-1}y^{-1}] $ was proven to be  surjective 
for every finite non-abelian simple group  (\cite{LOST}) and for $G=\SL(n,k)$  if the field $k$ contains at least 4 elements (\cite{Thp})

The geometric 
and arithmetic properties of $S_{g}\subset \SL(2,k)^2$ were studied in \cite{BKS} (k being a number field)
for words belonging to the first derived subgroup  $[\CF_2,\CF_2].$
It appeared that these properties mostly depend on  the trace $\al:=tr(g)$ and  the trace polynomial   $P_w(s,t,u)=\tr(w(x,y)) $ of a word $w(x,y).$ 
Here $s=\tr(x),t=\tr(y),u=\tr(xy).$  Actually the main role is played by the {\sl trace hypersurface} $$H_{w,\al}=\{P_w(s,t,u)=\al\}\subset \BA^3_{s,t,u}.$$

 More precisely, in  \cite{BKS})   a certain conic bundle  over the three  dimensional affine space   $\BA^3_{s,t,u}$ is constructed. It does not depend on a word and reflects the intrinsic structure of the group.
For a fixed word $w$ and a fixed $g\in \SL(2,k) $  with   $\al:=tr(g)$  the solution set $S_{g}$ is proven to be isomorphic to an open dense subset of the restriction of  this conic bundle onto $H_{w,\al},$
provided $w\in[\CF_2,\CF_2].$

In this paper we  consider the case $k=\BC.$   We use the geometric description of the set $S_{w,g}$  as specific  conic bundle,  given in  \cite{BKS},  to 
study the   automorphism group  $\Aut(S_{g})$ for certain words $w.$
We compute  the  Makar-Limanov invariant   $\ML(S_{g})$   and show   that  $\Aut(S_{g})$ has  the Jordan property.
 
 The Makar-Limanov invariant  (denoted as $\ML$ or, sometimes,  $\mathrm{AK}$)     was introduced  by L. Makar-Limanov (see,   e.g., \cite{ML},
 \cite{ML96}) as a tool to distinguish 
 a contractible affine variety from an affine space.  
  The idea was that an affine space is covered by the families of straight lines and admits many $\BC^+$- actions. $\ML$-invariant was used to establish  that the famous Koras-Russell threefolds are not isomorphic to $\BA^3,$ and thus in proving the linearization conjecture \cite{KKML}. 
 
 \begin{Definition}\label{MLdef}  Let $X$ be an  an affine variety with the ring of regular  functions $\CO(X) .$ Then $\ML(X)\subset \CO(X) $ is the subring of  all those regular functions 
$f\in\CO(X)$ that are invariant under all $\BC^+$ -actions on $X$.

\end{Definition} 

 For an affine variety $X$ the group $\Aut(X)$ contains no subgroup isomorphic to $\BC^+$ if and only if $\ML(X)=\CO(X).$  If $\ML(X)=\BC$  then  the group  $\Aut(X)$   has a dense orbit.
  We provide the algebraic definition  and some details  in \secref{ML},
and  prove the following.
 \begin{Theorem}\label{Intro:Main}   (\thmref{mlcor:trivial}  ).
 Let  $n,m$ be non-zero integers, $w(x,y)=[x^n,y^m]$ and  let $g\in\SL(2,\BC)$  have trace $\al\ne\pm 2.$ 
Let $$S_g:=\{(x,y)\in \SL(2,\BC)\times \SL(2,\BC) \ | \ w(x,y)=g\}$$ 
Then the Makar-Limanov invariant $\ML(S_g)=\CO(S_g),$  the ring of regular functions on $S_g. $  In other words, $S_g$ admits no $\BC^+-$actions.\end{Theorem}

 Next we  consider the following  Jordan  properties of a group.
\begin{Definition}\label{groups} \begin{enumerate}  
 \item  A group $\Gamma$ is called {\sl Jordan} if there is a positive integer $J$ such that
every finite subgroup $B$ of $\Gamma$ contains an abelian subgroup $A$  that is normal in $B$ and such that the index $[B:A]\le J$ (\cite[Question 6.1] {Serre1}, \cite[Definition 2.1]{Pop}).
\item  A group  $\Gamma$  is  {\sl very Jordan} (\cite{BZ20})   if there exist a   commutative normal subgroup  $\Gamma_0$ of $\Gamma,$ an integer $n>0,$ and a  group $F$  such that  \begin{itemize} \item $F$ sits in  a short exact sequence
\begin{equation}\label{veryjordan}
1\to \Gamma_0\to\Gamma \to F\to 1.\end{equation}
\item every finite subgroup of $F$ has at most $n$ elements.\end{itemize}
\end{enumerate}\end{Definition}

 The Jordan property  for  groups of birational and biregular automorphisms of algebraic varieties
 (as well as holomorphic automophisms of complex manifolds or diffeomorphisms of smooth varieties) was  actively studied  recently. One can find the 
 details in,  for example,  surveys    \cite{BZ22},  \cite{BZ24},   \cite{LMZ}.  
   It is known, e.g.,  (see also \propref{jordanprop:known}) that if $X$ is a projective variety, then\begin{itemize}\item 
      the automorphism group $\Aut(X)$  is Jordan (\cite {MengZhang});
\item the group of birational  selfmaps $\Bir (X)$ is Jordan provided  $X$ is    non-uniruled  (\cite{PS14}); 
\item    $\Bir ( A\times \BP^n)$ is not Jordan, where $A$ is an abelian variety of positive dimension  (\cite{Zar14}).
\end{itemize}

  The group $\Aut(X)\subset \Bir X$ for a quasiprojective variety $X$  is less studied. It is known that
  $\Aut(X) $  is Jordan if either $\dim(X)\le 2$   (\cite{BandmanZarhinTG}), or $\dim(X)=3 $ and $X$ is not birational to the product  $E\times \BP^2$ of an elliptic curve $E$ with the projective plane (\cite{BZopen}).  There is no example of an algebraic variety with non-Jordan automorphism group.

  The geometric description of $S_{g}$ as a  conic bundle 
 over the trace hypersurface permits to obtain the following 

\begin{Theorem}\label{Intro:Main1}(\thmref{jordanthm:jordan}).  Let  $n,m$ be non-zero integers.  Let  $\al\ne\pm 2$ be a complex number, and $g\in \SL(2,\BC) $ a  matrix with     $\tr(g)=\al.$ Then  \begin{enumerate}
\item   If $ w\in \CF_2 , \ w=[x^n,y^m],$     then $\Aut(S_{g})$ is Jordan;
\item  If $ w\in \CF_2 , \ w=[x^n,y^m],$  $|n|>2$ and  $|m|>2,$ then $\Aut(S_{g})$  is very Jordan\end{enumerate}\end{Theorem}


 In  \secref{trace} we provide properties of the trace polynomials and trace hypersurfaces.   In  \secref{conic} we  describe  the construction of the conic bundle and the geometric structure of the solution variety $S_{g}$ following \cite{BKS}. In  \secref{ML} we compute the $\ML$-invariant of $S_{g}$, and in \secref{jordan} we  establish Jordan properties for $\Aut(S_{g}),$  for certain words $w.$ 
\vskip 0.2 cm

{\bf Notation and assumptions 1.}\begin{enumerate}
\item $\BN,\BZ,\BQ,\BR,\BC$ stand for, respectively, the set   of natural numbers and ring of  intergers, and for the fields of rational, real and complex numbers.  $\BC^+,\BC^*$are  the additive and multiplicative groups of complex numbers,  respectively.

\item $G=\SL(2,\BC),$ and $Id\in G$ is the identity matrix.
\item $\tr(x)$ is the trace of element $x\in\SL(2,\BC).$
\item $\CF_2$ is the  free group on two generators; $w\in \CF_2$ is called  a {\sl word in two letters} and the corresponding map $G^2\to G$ is called {\sl the word map}.
 \item $\Tr:G^2\to\BA^3_{s,t,u}$ is   the map $(x,y)\in G^2\mapsto (s,t,u) , $
 where   $s=\tr(x),  \ t=\tr(y), u=\tr(xy).$
\item $\BA^n_{w_1,\dots,w_n}, \BP^N_{w_0:\dots :w_n}$  are, respectively, affine  and projective spaces with coordinates displayed  in the subscript.
\item If $X$ is an affine variety, then $\CO(X)$ and $\BC(X)$   stand for the ring and the field of regular and rational functions,  respectively, on $X.$ 
\item     $\ov{\kappa}(X)$ stands for the logarithmic Kodaira dimension  of $X.$
\item The sign $\cong$ means an isomorphism, $\sim$ means linear equivalence of divisors.
\item $K(X)$ stands for the canonical class    of a variety $X.$
\item $\Aut(X)$ and $\Bir(X)$ are the groups of biregular and birational automorphisms of $X$, respectively.
\item  We say that an assertion is valid for  the general point of a variety $X$ if it is valid for  any point  belonging to an open subset $U\subset X.$
 The general fiber of a morphism $f:X\to Y$  is  the  fiber $f^{-1}(y)$ over any point $y$ belonging to an open subset $U\subset Y.$

\end{enumerate}

{\bf Aknowledgments.} The author is most grateful to Sh. Kaliman, B. Kunyavskii, and Yu.  G. Zarhin for interesting and helpful discussions. 
 My special  thanks  go to Prof. Osamu Fujino for  valuable information related to logarithmic Kodaira dimension.

\section{The Trace map} \label{trace}

  In this section we recall the definition of the trace map for $G=\SL(2,\BC)$ and provide its 
  properties for words we are interested in. The trace map is the classical tool, going back to Vogt and  Klein.
 
Any  word 
$w\in \CF_2$, 
induces the trace map $\Tr_w:\SL(2,\BC)^2\to \BC$
defined by 
$\Tr_w(x,y)= \tr (w(x,y))$.
 We will use the following special case of the general  theorem on  trace map .

\begin{Theorem}\label{Hor} (see, e.g, \cite[Theorem A]{Gol}  and reference therein).
For a word  $w\in \CF _2$, there exist   a trace polynomial $P_w(s,t,u)\in\BZ[s,t,u]$  such that 
$$
\tr(w(x,y))=P_w(s,t,u)
$$
whenever   $x,y\in\SL(2,\BC):=G$ and $s=\tr(x), \ t=\tr(y), \ u=\tr(xy).$
\end{Theorem}

 It follows from \thmref{Hor} that
the following diagram commutes:

\begin{equation} \label{traceeq:cd}
\begin{CD}
G^2   @>{w}>>  G  \\
@VTrVV @V\tr VV \\
\BA^3_{(s,t,u)} @>{P_w} >> \BA^1.
\end{CD}
\end{equation}
For  $G=\SL(2,\BC)$  the projection $\Tr\colon G\times G\to \BA^3_{s,t,u}$
is surjective.
Note that there are non-conjugate words with the same trace polynomials. For example,
$\tr a^2bab^{-1}=\tr a^2b^{-1}ab$ though theses words are not conjugate  (\cite{Ho}).
\begin{Example}\label{traceex:com} Let $w(x,y)=[x,y]=xyx^{-1}y^{-1}, \ s=\tr(x), t=\tr(y), u=\tr(xy)$. Then 
$$P_w(s,t,u)=s^2+t^2+u^2-stu-2:$$ 
 The trace  hypersurface
$$H_{w,\al}=\{\tr([x,y])=\al\}=\{s^2+t^2+u^2-stu-2=\al\}\subset \BA^3_{s,t,u}$$
is  irreducible for all $\al.$   If $\al^2\ne 4$,  it is smooth.\end{Example}

 Further on we will use the following notation:
 
{\bf Notation and Assumptions 2}\begin{enumerate}\item For $\al\in\BC$ and $P\in\BC[s,t,u]$
$$H_{P,\al}:=\{P(s,t,u)=\al\}\subset \BA^3_{s,t,u};$$
\item For $\al\in\BC$ and $w\in\CF_2$
$$H_{w,\al}:=H_{P_w,\al}=\{P_w(s,t,u)=\al\}\subset \BA^3_{s,t,u};$$
\item
$$J(s,t,u):=P_{[x,y]}(s,t,u)-2=s^2+t^2+u^2-stu-4=$$ $$\left(u-\frac{st}{2}\right)^2-\frac{(t^2-4)(s^2-4)}{4};$$
\item  $$M:=\{P_{[x,y]}=2\}=\{J(s,t,u)=0\};$$
\item    $$V=\BA^3_{s,t,u}\setminus M, \ \tilde V=\Tr^{-1}(V)\subset G\times G;$$
\item   $D_k(s):=D_{k}(s,1),  E_{k}(s):=E_{k}(s,1)$ are the Dickson polynomials of the first and the second kind, respectively (see \defnref{trace:dickson}).
\item 

$\phi_k(s)=E_{k+1}(s,1).$ 
\end{enumerate}

{\bf  Properties of the  trace map  on $\SL(2,\BC)$}\begin{enumerate}\item $\Tr(x,y)\in M$ if and only if   matrices $x,y$ have a common eigenvector. In particular,  there exist $z\in G$ such that both $zxz^{-1}$ and  $zyz^{-1}$ are upper triangular  (\cite[Theorem 2.9]{Fri83}).
We will use lower triangular matrices instead of upper ones. 
\item  If $w=x^ny^m ,  \  n,m\ne 0$   then \begin{equation}\label{traceeq:ab}\tr  w(x,y)= u\cdot
f_{n,m}(s,t)+h_{n,m}(s,t), \end{equation} where
$f_{n,m}(s,t), \ h_{n,m}(s,t)\in \BZ[s,t]$  and 
the highest degree summand of $f_{n,m}(s,t)$ (of degree $n+m-2$) is:
$\pm s^{n-1}t^{m-1}$ (\cite[Lemma 2.3]{BG})
\item
Assume that   $n_i,m_i\ne 0, \ w_i(x,y)=x^{n_i}y^{m_i}$  and  the word $$w=x^{n_1}y^{m_1}\dots
x^{n_k}y^{m_k}=w_1\dots w_k.$$  Then 
\begin{equation}\label{traceeq: word}
\tr(w)=P_w(s,t,u)=\sum\limits_{0}^k u^r G_r(s,t) \text{ and }
G_k(s,t)=\prod\limits_{i=1}^{k}f_{n_i,m_i}.\end{equation}
 (\cite[Lemma 2.3] {BG})
 \item Assume that $w\in[\CF_2,\CF_2].$ Then \begin{equation}\label{traceeq:comword}
P_w(s,t,u)-2=Q_w(s,t,u)J(s,t,u)\end{equation} for some polynomial $Q_w\in \BZ[s,t,u]$  (\cite[Lemma 3.1] {BKS}).
 \end{enumerate}

\begin{Definition}\label{trace:dickson}(\cite [Section 9.6]{MP}). For $k\in\BN$ the Dickson polynomials $D_k(s,a)$ and  $E_k(x,a), \ x,a\in\BC$ of the first and the second kind,
 respectively. are defined by the following recursive rule:
$$D_0(x,a)=2, \ D_1(x,a)=x, \ D_k(x,\al)=x D_{k-1}(x,a)- a D_{k-2}(x,a),$$
$$E_0(x,a)=1, \ E_1(x,a)=x, \ E_k(x,a)=x E_{k-1}(x,a)- a E_{k-2}(x,a).$$\end{Definition}

\begin{Proposition}\label{tracelem:ab}   For $k\in\BN, k>0$  denote   $\phi_k(s)=E_{k-1}(s,1).$ 
Assume  that $ n,m$ are non-zero integers. Then
\begin{enumerate}\item $f_{n,m}=\pm \phi_{n}(s)\phi_{m}(t),$ where $f_{n,m}$ are defined  in \eqnref{traceeq:ab};
\item if $w(x,y)=[x^n,y^m]$  then $Q_w=(\phi_{n}(s)\phi_{m}(t))^2,$ where $Q_w$ was  defined  in \eqnref{traceeq:comword}  .
\end{enumerate}\end{Proposition}
\begin{proof} We use the following properties of  the trace on $G.$
\begin{equation}\label{traces}
\tr(x)=\tr(x^{-1}),    \                \tr(xy)=\tr(yx),  \ \tr(xy)+\tr(xy^{-1})=\tr(x)\tr(y).\end{equation}

 We prove (1) by induction.

{\bf  Case 1.}  $m=1, n>0, w(x,y)=x^{n}y.$  We have 

$\tr(xy)=u,$  thus $f_{1,1}(s,t)=\phi_1(s)=1.$

$\tr(x^2y)=su-t,$  thus $f_{2,1}(s,t)=\phi_2(s)=s.$

$\tr(x^ny)=s\tr(x^{n-1}y)-\tr(x^{n-2}y),$  thus $f_{n,1}(s,t)=sf_{n-1,1}(s,t)-f_{n-2,1}(s,t)=s\phi_{n-1}(s)-\phi_{n-2}(s)=\phi_n(s).$

{\bf  Case 2.}  $m=1, n>0, w=x^{-n}y.$ We have  

 $\tr(x^ny)+\tr(x^{-n}y)=t\tr(x^{n})=tD_n(s)$, thus   $f_{-n,1}=-f_{n,1}.$ 

{\bf  Case 3.}  $n\ne 0, m\ge 1, w=x^{n}y^m.$  We have 

 $f_{n,1}(s,t)=\sg(n)\phi_n(s)$ according to {\bf Cases 1, 2.}

$\tr(x^ny^2)=t\tr(x^ny)-tr(x^n), $  thus $f_{n,2}(s,t)=tf_{n,1}(s,t)=\sg(n)\phi_2(t)\phi_n(s).$

$\tr(x^ny^m)=t\tr(x^ny^{m-1})-\tr(x^{n}y^{m-2})$  thus 

$f_{n,m}(s,t)=\sg(n)(t\phi_n(s)\phi_{m-1}(t)-\phi_n(s)\phi_{m-2}(t))=\sg(n)\phi_n(s)\phi_{m}(t).$

{\bf  Case 4.}  $n\ne 0, m\ge 1, w=x^{n}y^{-m}.$  We have 

$\tr(x^ny^{-m})=\tr(y^mx^{-n})=\tr(x^{-n}y^{m})=\pm\phi_n(s)\phi_{m}(t).$

Proof of (2).  In notation of {\bf Properties}(3)(4),  we have 

$$ (\phi_n(s)\phi_m(t))^2u^2 +uG_1(s,t)+G_0(s,t)=Q_w(s,t,u)(u^2-ust+t^2+s^2-4),$$ 
hence
$$Q_w(s,t)=(\phi_n(s)\phi_m(t))^2.$$
\end{proof}
\begin{Remark}\label{tracerem:positive}
 Since   $  \tr(x^{-1})=\tr(x)=s, \ \tr(x^{-1}y^{-1})=\tr(xy)=u, \ \tr(xy^{-1})=st-u,$ and $J(s,t,u)=J(s,t,st-u),$  for any integers $n,m$ 
 $$\tr([x^n,y^m]=\tr([x^{|n|},y^{|m|}].$$ \end{Remark}

\begin{Corollary}\label{tracecor:rational}  Assume that  $n,m$ are  non-zero integers,  a complex number $\al\ne\pm 2,$ and  either $w(x,y)=x^ny^m$  or  $w(x,y)=[x^n,y].$ Then the trace hypersurface
$H_{w,\al}=\{P_{w}(s,t,u)=\al \}$ is rational.\end{Corollary}

\begin{proof}  (1)   Assume that $w(x,y)=x^ny^m.$  Then 
$$H_{w,\al}=\{\epsilon\phi_n(s)\phi_{m}(t)u+h_{n,m}(s,t)=\al\},$$  where $\epsilon=\pm 1.$  The
map $$(s,t)\mapsto(s,t,u=\epsilon\frac{\al-h_{n,m}(s,t)}{\phi_n(s)\phi_{m}(t)})$$
establishes the needed birational  map from $\BA_{s,t}$ to $H_{w,\al}.$

(2) Assume that $w(x,y)=[x^n,y].$ Then 
\begin{equation}\label{trace:1}
H_{w,\al}=\{(\phi_n(s))^2J(s,t,u)=\be^2\},\end{equation}  where we define $\beta$ by $\beta^2=\al-2.$

Since $\phi_n(s)$ does not vanish on  $H_{w,\al}$ we rewrite \eqnref{trace:1} as \begin{equation}\label{trace:2}
(u-\frac{st}{2})^2-\frac{(t^2-4)(s^2-4)}{4}-(\frac{\be}{\phi_n(s)})^2=0;\end{equation}.

Define 
\begin{equation}\label{trace:3}
r(s,t,u)=2\frac{u-\frac{st}{2}-\frac{\be}{\phi_n(s)}}{(t-2)(s-2)}=
\frac{1}{2}\frac{(t+2)(s+2)}{u-\frac{st}{2}+\frac{\be}{\phi_n(s)}}\end{equation}.

 Fix values of $s\ne\pm 2$ and $r\ne 0.$  Then 
\eqnref{trace:3} gives 
 $$t=\frac{4\frac{\be}{\phi_n(s)}-2(\frac{s+2}{r}+r(s-2))}{\frac{s+2}{r}-r(s-2)}$$
$$4(u-st)=r(t-2)(s-2)+\frac{(t+2)(s+2)}{r}.$$

This  defines a birational map $\BA^2_{r,s}$ to $H_{w,\al}.$
\end{proof}

  For $n\ge 1, m\ge 1$ let 
 $$A_n=\{s | (s^2-4)\phi_n'(s)+s\phi_n(s)=0\}\subset\BC$$

$$B_{n,m}=\{\al\in\BC \ | \ \al-2=-(\phi_n(s_i)\phi_m(t_j))^2\frac{(s_i^2-4)(t_j^2-4)}{4}, s_i\in   A_n, \ t_j \in A_m\}$$
 These  sets    are  finite, $A_n$  and $B_{n,m}$  contain at most $n+1$ and   $(n+1)(m+1)$  elements, respectively.

 \begin{Lemma}\label{trace:singular}   Let  $\al\ne\pm 2$ be a  complex number, and $n,m$ be non-zero integers, and 
 $w(x,y)=[x^n,y^m].$ 
 Then \begin{itemize}\item if  the trace surface  $H:=H_{w,\al}$  has singular points, then $\al\in  B_{n,m};$
\item the set of singular points  of $H$ is  either empty or finite;
 \end{itemize}
 \end{Lemma}
 \begin{proof}

 Let $P=P_w$ and  let $(s_0,u_0,t_0)$ be a singular point of $H.$ Then  the partial derivatives $P_s,P_t  ,P_u$
vanish at this point. The direct computation shows that
$$P_u=(\phi_n(s)\phi_m(t))^2(2u-st).$$  Since $\phi_n(s)\phi_m(t)$ does not vanish on $H$, we have
\begin{equation}\label{trace:derivativeu}2u_0-s_0t_0=0,  \  J(s_0,u_0,t_0) =-\frac{(s_0^2-4)(t_0^2-4)}{4}.\end{equation}

 Computing $ P_s,$ and using \eqnref{trace:derivativeu}   one gets:
$$(\phi_m(t_0))^2[2\phi_n(s_0)\phi_n' (s_0)J((s_0,u_0,t_0) +\phi_n(s_0)^2(2s_0-u_0t_0)]=$$ 
 $$(\phi_m(t_0))^2\phi_n(s_0)[2\phi_n' (s_0)(-\frac{(s_0^2-4)(t_0^2-4)}{4})+\phi_n(s_0)(2s_0)(1-\frac {t_0^2}{4})]=$$  $$-\frac{(t_0^2-4)}{2}(\phi_m(t_0))^2\phi_n(s_0)[\phi_n' (s_0)(s_0^2-4)+\phi_n(s_0)(s_0)]=0\Longrightarrow $$ $$ \phi_n' (s_0)(s_0^2-4)+s_0\phi_n(s_0)=0.$$
 Note that    $ (t_0^2-4)\ne 0$ since $J(s_0,t_0,u_0)\ne 0.$
 
 Similarly, one gets:$$\phi_m' (t_0)(t_0^2-4)+t_0\phi_n(t_0)=0.$$  It follows that $s_0\in A_n, t_0\in A_m,$  and
$$\al-2=(\phi_m(t_0))^2\phi_n(s_0)^2J(s_0,u_0,t_0)\in B_{n,m}.$$

 Since the sets 
$A_n$ and $A_m$ are finite, the set of singular points on $H$ is finite as well.\end{proof}

\begin{Proposition}\label{traceprop:irreducible} 
Let  $\al\ne\pm 2$ be a  complex number, and $n,m$ non-zero integers,  $w(x,y)=[x^n,y^m].$ 
 Then  the trace surface  $H:=H_{w,\al}$ is irreducible.\end{Proposition}

\begin{proof}  

 The equation of $H$ implies that  
  on $H$ 
   $$(\phi_n(s)\phi_m(t))^2(2u-st)^2=4(\al-2) +(\phi_n(s)\phi_m(t))^2(t^2-4)(s^2-4),$$
 i.e.,   $H$ is a double  cover of an open subset $U:=\{\phi_n(s)\phi_m(t)\ne 0\}$ of an affine plane
 $\BA_{s,t}.$   Moreover, since every irreducible component of $H$ has to have dimension 2, and the map is finite,
there may be at most  two irreducible components. If there are two of them, then  both  are  mapped 
on $U$  dominantly and  birationally.
These components should  meet at points of the curve $$ C:=\{(2u-st=0\}\cap H$$ 
   which would imply that $H$ has infinite set of singular points. But this contradicts to \lemref{trace:singular}.
   Thus $H$ is irreducible.
\end{proof}

 \begin{Remark}\label{kodaira-inequalities}
 We will use the following properties of the  logarithmic Kodaira dimension of a quasiprojective manifold  $W$.  \begin{enumerate}\item  Assume that $f:V\to W$ is a dominant morphism of affine  varieties. 
Let $F$ be an irreducible component  of the (sufficiently) general fiber of $f.$
Then \begin{itemize}\item 
 \begin{equation}\label{trace:101}
 \ov{\kappa}(V)\le \ov{\kappa}(F)+\dim W;
\end{equation}
( {\sl the easy Iitaka inequality},   (\cite[Theorem 4]{Iit77}). 
\item   \begin{equation}\label{intro:102}
 \ov{\kappa}(V)\ge \ov{\kappa}(F)+ \ov{\kappa} (W).
\end{equation} 
(subadditivity of the logarithmic Kodaira dimension,  the inequality due to Osamu Fujino\cite[Corollary 5.3.2]{Fuj}).
 In our case the general fibers (over an open subset of $W$)  are irreducible and have the same log-Kodaira dimension, thus are ``sufficiently'' general.
\end{itemize}

\item      If $W$ is a quasiprojective complex manifold   and $\dim W=\ov \kappa(W)$ then the group $\Aut(W)$ is finite  (\cite[Theorem 5.2]{Sakai} ).   

If $W$ is singular, this property still holds since singular locus is invariant under any automorphism.

\item  Assume that we have a surjective morphism $f$   from an affine surface $V $ to a smooth affine surface $W$. Then $\ov \kappa(V)\ge\ov \kappa(W).$ provided $\kappa(W)\ge 0$ (\cite[Theorem 14]{Iit78}).
 This fact follows directly also from \cite[Equation (R)]{Iit77}.
\item  Let  $W$ is an algebraic variety equipped with the action of a connected algebraic group,  let  $O_p$ be an orbit of a general point $p\in W$.
Then   \begin{equation}\label{Iitaka1}
 \ov{\kappa}(W)\le \ov{\kappa}(O_p)+\dim (W)-\dim(O_p).
\end{equation} 
(\cite[Lemma 2]{Iit77}).
\end{enumerate}

\end{Remark}

\begin{Proposition}\label{traceprop:kappa} 
Let  $\al\ne\pm 2$ be a  complex number, and $n,m$ non-zero integers, $w(x,y)=[x^n,y^m].$  
Then for the logarithmic Kodaira dimension $\ov{\kappa}(H)$
of the trace surface  $H:=H_{w,\al}$ the following hold:\begin{enumerate}
\item if  $(n,m)=(\pm 1,\pm 1), $ then $\ov{\kappa}(H)=0;$
\item   $\ov{\kappa}(H)\ge 0;$

\item   if  $2<|n|$  and $|m|=1$ or  $2<|m|$  and $|n|=1,$ then $\ov{\kappa} (H)\ge 1 ;$ 
\item if $2<|n|$  and $2<|m|,$ then 
$\ov{\kappa}(H)=2.$

\end{enumerate}\end{Proposition}

\begin{proof} Thanks to \remarkref{tracerem:positive} we may  assume that $n>0,$ $m>0$  and $n\ge m.$

{\bf Case 1.} $(n,m)=( 1,1) .$  In this case the projective closure $\ov H$ of $H$ in the projective space $\BP^3_{\tilde s:\tilde t:\tilde u:z}$  where $s=\frac{\tilde s}{z},t=\frac{\tilde t}{z},u=\frac{\tilde u}{z}$,   is 
$$\ov H=
(z(\tilde s^2+\tilde t^2+\tilde u^2)-\tilde s\tilde t\tilde u-2z^3=\al z^{3}\}.$$
  $\ov H$ is a smooth projective surface and $H=\ov H\setminus D,$ where the divisor
$D=\{z=0\}\cap\ov H=\{\tilde s\tilde t\tilde u=0\}$  has strictly normal crossings and $K(\ov H)+D\sim 0.$ Hence $\ov{\kappa}(H)=0.$

{\bf Case 2.}   Denote the surface considered in the previous item by $H_{1,1}$
For any $(n,m)$ there is a generically finite  dominant map $\Phi_{n,m}:H\to H_{1,1}$  defined by 
$$\Phi_{n,m}(s,t,u)=(D_n(s),D_m(t), P_{n,m}(s,t,u)),$$
where 
\begin{itemize}\item
$D_k(s)$ is the  $k^{th}$ Dickson polynomial of the first kind (the trace  polynomial of the word  $w(x,y)=x^n$);
\item $P_{n,m}(s,t,u)$ is the trace  polynomial of the word   $w=x^ny^m.$\end{itemize} 

 Since  logarithmic Kodaira dimension does not 
 increase under generically finite morphism (see item 3 of \remarkref{kodaira-inequalities}), we have $$\ov{\kappa}(H)\ge\ov{\kappa}(H_{1,1})= 0.$$

{\bf Case 3.}  Consider the function  $ p_s:H\to \BC, \ (s,t,u)\mapsto s.$ If $n>2$ then it omits   infinity and $n-1>1$ finite  values : the roots of $\phi_n.$ Thus  $\ov\kappa( pr_s(H))=1.$

 For a general  $a$ the  fiber $p_s^{-1}(a)$ is a plane curve $$C_a=\{(\phi_n(a))^2(u^2-aut+t^2+a^2-4)=\al-2\}$$ is smooth irreducible, and  has  two punctures.  Hence, $\ov{\kappa}(C_a)\ge 0.$
  The inequalities from \remarkref{kodaira-inequalities}
   imply:
 \begin{equation}\label{trace:1010}
 \ov{\kappa}(H)\le \ov{\kappa}(C_a)+\dim pr_s(H)=\ov{\kappa}(C_a)+1,
\end{equation} 
  \begin{equation}\label{trace:102}
 \ov{\kappa}(H)\ge \ov{\kappa}(C_a)+ \ov{\kappa} (pr_s(H))= \ov{\kappa}(C_a)+ 1.
\end{equation}

{\bf Case 4.}   Let $N:=\BC_{s,t}\setminus \{\phi_n(s)\phi_m(t)=0\}.$ Let $pr_s:N\to  \BC_s$ be the projection  to $\BC_s.$ Then the image is $\BP^1$ with $n$ punctures and the fiber is $\BP^1$ with $m$ punctures, thus the inequalities from \remarkref{kodaira-inequalities} give $\ov{\kappa}(N)=2.$
Consider the projection $pr:H\to\BC_{s,t},  \ \ pr(s,t,u)=(s,t).$ The map is dominant and the image 
$pr(H)\subset N.$ Hence, 
 
$$\ov{\kappa}(H)\ge\ov{\kappa}(N)=2.$$
\end{proof}

\begin{Lemma}\label{elliptic} 
Let  $\al\ne\pm 2$ be a  complex number, and $n\ge 2,m\ge 2$ integers, $w(x,y)=[x^n,y^m] $ and $H:=H_{w,\al}.$  
Assume that $H$ is birational to the direct product $E\times\BP^1, $ where $E$ is an elliptic curve. 
 Let $pr:H\to E$ be  the induced  projection.  Then one of the following hold:
\begin{itemize}\item[a)]  $pr:H\to E$ is not surjective;
\item[b)]  the general fiber $F_e=pr^{-1}(e), e\in E$ is isomorphic to $\BP^1 $ with $k\ge 3$ punctures.
\end{itemize}\end{Lemma}
\begin{proof}   Assume that b) does not hold. It means that $F_e$ has at most two punctures.  Since functions $s$   and $t$ omit at least two values on $H,$ then either, say,  $s \mid_{F_e}$ is constant, or it has pole and zero precisely at punctures of $F_e,$ (in particular, there are precisely two punctures).

If $s$ is constant along $F_e$ for a general $e$ then  there is a rational function  $h\in\BC(E)$ such that $s=pr^*(h).$ 

If both $s$ and $t$ are not constant along $F_e$ then they both  have poles and zeros at punctures of $F_e$, thus there is a rational function
$h\in\BC(E)$ and integers $k,l$ such that $s^kt^l=h(e).$ Since the map $H\to\BC_{s,t}$ is dominant, $h$ cannot be constant.

 In both cases  the fibers over poles and zeros of $h$ do not intersect $H$, thus  $pr(H)\ne E. $
\end{proof}

\section{Conic bundle, representation of the solution set $S_{w,g}$} \label{conic}

In this section we briefly describe the construction 
 introduced in \cite{BKS}.  For readers convenience we provide brief but self-contained and explicit exposition.
   The construction  contains three  elements. \begin{enumerate}
 \item A  $\BP^1-$ bundle $X$  which is a conic bundle
over $V=\BP^3\setminus M$ (see {\bf Notation 2}). It does not depend on a word and is intrinsic part of $\SL(2)^2$ structure.
\item For a given word $w$ and a given element $g\in \SL(2)$  with $\tr(g)=\al\ne \pm 2$ the restriction $X_\al:=X_{w,\al}$ of the conic bundle $X$ 
onto the trace hypersurface  $H_{w,\al}.$ 
\item Embedding of $S_{w,g}$ into  $X_{w,\al}$  as an open dense subset, provided that $w\in[\CF_2,\CF_2]$ and  $H_{w,\al}$ is irreducible. 
\end{enumerate}

  {\bf Notation  and  assumptions 3}
\begin{enumerate}
\item $B\subset  G$- the subgroup of all lower triangular matrices.
\item $B-$ action  means action by conjugation  with elements of $B$ on $\SL(2,\BC).$
\item $T(s,t,u) =\mathrm{Tr}^{-1}(s,t,u)\cap  T  $ for a subset $T\subset G\times G$ and a point $(s,t,u)\in\BA^3.$

\item For $t\in\BC$ we set 
\begin{equation} \label{yt}
y_t=\begin{bmatrix} t & 1 \\ -1 & 0 \end{bmatrix},
\end{equation}

\item $Z_G(y)$ stands for the centralizer of $y\in G.$  If $y=y_t$ then  $Z_G(y)$ 
consists of all  matrices  \begin{equation}\label{eq:centr}
C(t,\gamma,\delta):=\begin{bmatrix}\frac{\gamma+t\delta}{2} & \delta\\-\delta&\frac{\gamma-t\delta}{2}\end{bmatrix}.
 \end{equation}
where $\ga,\delta\in\BC$ are subject to the condition  (reflecting $\det (C(t,\gamma,\delta))=1$)
\begin{equation}\label{centrcondition}
\gamma^2-(t^2-4)\delta^2=4.\end{equation}
It follows that $Z_G(y_t)\cap B=\pm Id.$ 
 Moreover, if $AC(t,\gamma,\delta)=C(t,\gamma,\delta)A$ for 
 $A\in G,$   then $A=C(t,\gamma',\delta')$ for some $\gamma',\delta'.$  This may be checked by direct computation.
\end{enumerate}

 \subsection{Construction of conic bundle $X$ }\label{wX}
 
{ \ } 

 {\bf  Step 1 .}   First consider set $Y\subset G\times G$  defined by  \begin{equation}\label{Y}
Y=\{(x,y_t)\in (G\times G)\cap \tilde V\}\subset G\times G.\end{equation}

    If the pair $(x,y_t)\in Y(s,t,u),$ then  
\begin{equation} \label{eq:xabP}
x:= x(a,b,s,t,u)=\begin{bmatrix}a & b \\ u+b-at & s-a\end{bmatrix}.
\end{equation}
Requiring that its determinant equals 1 means that $$\det x(a,b,s,t,u)=a(s-a)-b(u+b-at)=1,$$
 which, for a  fixed $(s,t,u)\in\BA^3$,  defines a conic  $R(s,t,u)$   in $\BA^2_{a,b},$ namely
\begin{equation} \label{eq:R}
Y(s,t,u)\cong R(s,t,u):=\{(a,b)\in\BA^2_{a,b}  \  | \ \ a(s-a)-b( u+b-at)=1\}.\end{equation}
 Thus there is an isomorphism $Y\to Y_X$ onto an affine variety 
 \begin{equation} \label{eq:yx}
 Y_X=\{ (s,t,u)\in V, a(s-a)-b( u+b-at)=1\}\subset\BA^5_{s,t,u,a,b}\end{equation}

{ \ }

 {\bf Step 2.} Set $X$  is 
the closure of $Y_X$ in $V\times \BP^2_{\tilde a:\tilde b:\tilde c},$  where  $a=\frac{\tilde a}{\tilde c},$  \ $b=\frac{\tilde b}{\tilde c.}$
 Namely  \begin{equation} \label{eq:x}
X=\{
-(\tilde a^2+\tilde b^2 +\tilde c^2) + t\tilde a\tilde b + s\tilde a\tilde c -  u\tilde b\tilde c = 0, (s,t,u)\in V\}\subset
 V\times \BP^2_{\tilde a:\tilde b:\tilde c}.\end{equation}
We denote by $p_{XV}$ the projection $(v,(\tilde a:\tilde b:\tilde c))\mapsto v$ of $X$ onto $V.$

 Let us mention the following properties of the construction. If $v=(s,t,u)\in V$ then \begin{itemize} 
 \item \begin{equation}\label{eq:tR}  p_{XV}^{-1}(v)=\ov   
 R(s,t,u) :=\{
-(\tilde a^2+\tilde b^2 +\tilde c^2) + t\tilde a\tilde b + s\tilde a\tilde c -  u\tilde b\tilde c = 0\}.\end{equation}
is a smooth rational curve;
 \item  $\ov   
 R(s,t,u) \setminus R(s,t,u)$ consists of   two points, if $t^2\ne 4,$ and a single point , if $t^2=4.$ 
\end{itemize}

 It follows that $X$ is a $\BP^1$ bundle over $V.$  

{ \ }

{\bf Step 3.}  We define a map $\pi:G^2\to X$ in the following way:

Let $(x,y)\in G^2, \Tr(x,y)=(s,t,u)\in V,$

\begin{equation} \label{xy} x=\begin{bmatrix} a_x& b_x \\ c_x & s-a_x \end{bmatrix},
y=\begin{bmatrix}a_y& b_y \\ c_y & t-a_y \end{bmatrix},
\end{equation}

 Then \begin{equation} \label{xy2}\pi(x,y)=(\Tr(x,y), (a_xb_y+ b_x(t-a_y):b_x:b_y)\in X\subset V\times\BP^2_{\tilde a:\tilde b:\tilde c}.\end{equation}

 The map $\pi $ is everywhere defined  in $\tilde V=\Tr^{-1}(V)$  because if $b_y=0,$ then $t-a_y\ne 0$ (since $\det(y)=1$) and $b_x\ne 0$ since $\Tr(x,y)\not\in M$ and $x,y$   cannot be  simultaneously conjugated to a pair of lower triangular matrices.  The following interpretation  was suggested by A. Skorobogatov.

\begin{Proposition}\label{quotient} $X$ is the  quotient of $G\times G$ with respect to  $B-$action. the action  by conjugation with the elements of the subgroup $B.$
\end{Proposition}

\begin{proof}   It is sufficient to prove that $\pi$ provides one-to-one correspondence between conjugate classes and points of $X.$ 

{\bf Case 1.}   Let $v=(s,t,u)\in V,\ \Tr (x,y)=v$  and $b_y\ne 0$   (in notation of  \eqnref{xy}) .
Then   after conjugation $(x,y)$  with 
\begin{equation}\label{z}   z(y)=\sqrt{b_y}^{-1}\begin{bmatrix}1&  0 \\ a_y-t&b_y \end{bmatrix}\in B,\end{equation}
 one gets the pair $(x(a,b,s,t,u),y_t),$ where 
\begin{equation}\label{ab}a=a_x+b_x(t-a_y)/b_y, b=b_x/b_y. \end{equation} 
   Thus 
    $\pi(x,y)=(\Tr(x,y), (a_xb_y+b_x(t-a_y):b_x:b_y):=\bf{x}\in X $ 
  and any pair with $b_y\ne 0$ and fixed $a$ and $b$  (see \eqnref{ab})   belongs to the $B-$action orbit of   
 $(x(a,b,s,t,u),y_t).$   
  Thus $\pi^{-1}({\bf x})$    is contained in the orbit.   
   
   On the other hand  for any lower triangular matrix 
  $$z=\begin{bmatrix}\lambda&  0 \\  r&\frac{1}{\lambda}\end{bmatrix},$$  
  define $x=z^{-1}x(a,b,s,t,u)z, \ y=z^{-1}y_tz.$
   Then $(x(a,b,s,t,u),y_t) $  and $\pi(x,y)=  {\bf x}.$
     It follows that   $\pi^{-1}({\bf x})$ contains  the   $B-action$  orbit  of $(x(a,b,s,t,u),y_t).$

Note that  $y$  with $b_y=0$ is not conjugate to $y_t$ by elements of $B.$ 

{\bf Case 2.}   $t^2\ne 4 ,  b_y=0.$   
Take roots of equation $t^2-t+1=0:$
$$\lambda_1=\frac{1}{\lambda_2}, \ \lambda_1+\lambda_2=t.$$
 For both $\lambda_i, i=1,2$ there exist precisely one  orbit of pairs   $(x,y)$ with $\Tr(x,y)=(s,t,u) $  that are $B-$conjugate to the pair
 
\begin{equation} \label{xy1} x=\begin{bmatrix} a& 1 \\ c & s-a \end{bmatrix},
y=\begin{bmatrix}\lambda& 0 \\ 0 & t-\lambda\end{bmatrix}\
\end{equation},

where $\lambda=\lambda_1$ or $\lambda=\lambda_2,$ and $$a=\frac{u\lambda-s}{\lambda^2-1}, \ c=a(s-a)-1, \ \pi(x,y)=(s,t,u)(1:\lambda:0)).$$

{\bf Case 3.}   $t^2= 4,  b_y=0.$  There is one orbit ($\lambda_1=\lambda_2$)  for $(s,2,s)$ and one for $(s,-2,s)$  that are mapped to points 
$(s,2,s)(1:1:0)\in X,$ and $(s,-2,s)(1:-1:0)\in X,$ respectively.
\end{proof}

\begin{Lemma}\label{absirr}
 Let $H\subset V$ be an irreducible hypersurface,  let $X_H=p_{XV}^{-1}(H).$ Then  $X_H$ is irreducible.
 \end{Lemma}

\begin{proof}  The restriction of $p_{XV}$ onto $X_H$  we denote by $p_H.$
 By construction, every irreducible component  $A$ of $X_H$ has dimension at least three. Since $H$ has dimension 2, and every fiber is $\BP^1,$
 we have $\dim A=3$ and it is mapped dominantly into $H.$  Hence $p_{H}(A)$  contains an open subset $U$ of $H.$ Moreover, every fiber $F=p_H^{-1}(u), 
  u\in U$ is irreducible  and  closed in $X_H.$  Hence $F\cap  A$ is one-dimension     closed  subset of $F,$ i.e., $F\subset  A.$  It follows that $A$ contains an open subset $p_H^{-1}(U).$
  
  It follows that  if 
  there were two components, they would intersect by an open subset  of $X_H,$ hence, coincide. \end{proof}  
  
  \subsection {Properties
 of $S_{w,g}.$}\label{subsec:Swg}

Let now $w\in[\CF_2,  \CF_2],$ be a word,  $P:=P_w$ its trace polynomial, $g\in G,$ $\al=\tr(g)\ne \pm 2$ 
and $$S_{w,g}=\{(x,y)\in G^2  \ | \ w(x,y)=g\}.$$

 Thanks to \propref{tracelem:ab}
 the hypersurface $$H:=H_{P,\al}=\{P(s,t,u)=\al\}\subset\BA^3$$ does not meet the surface $M=\{J=0\},$
thus $H_{P,\al}\subset V.$    The set $S_{w,g}$ is the subset of 
$ \mathrm{Tr}^{-1}(H_{P,\al}).$ 
We consider the  restriction $X_{P,\al}:= X_H=p_{XV}^{-1}(H_{P,\al})$  of the conic bundle $ p_{XV}: X\to V$ onto 
$H_{P,\al}$.   Namely, 
$$
X_{P,\al}=\{(s,t,u)(a:b:c)\in V\times\BP^2_{a:b:c} \ | \  P_w(s,t,u)=\al,$$ $$ -(\tilde a^2+\tilde b^2 +\tilde c^2) +
 t\tilde a\tilde b + s\tilde a\tilde c -  u\tilde b\tilde c = 0.\}$$

 We have a commutative diagram.

\begin{Diagram}\label{4d}
 \begin{alignat}{10}
     & G^2   &\supset\quad & S_{w,g}&     \notag\\
     &\downarrow  {\pi}   &      {}      & \downarrow  {\pi}  &  \notag \\
     & X           & \supset \quad&X_{P,\al}& \notag \\
   &\downarrow  {p_{XV}}   &      {}      & \downarrow  {p_{XV}}  &  \notag \\
  & V          & \supset \quad &H_{P,\al}& \notag 
\end{alignat}
\end{Diagram}

 The following    facts are  proven in  \cite[Theorem  3.4]{BKS}.

\begin{Proposition}\label{prop:main}(\cite[Theorem3.4]{BKS}) Assume that $w(x,y)\in[\CF_2,\CF_2] ,\tr(g)=\al\ne\pm2$ and $ H_{P,\al}$ is 
irreducible.  Then\begin{itemize}\item[(a)] $ X_{P,\al} $  is irreducible;
\item[(b)]
 $\mathrm{Tr}(S_{w,g})= H_{P,\al};$
 \item[(c)] The restriction  of $\pi$
on 
$S_{w,g}$ is injective dominant birational  morphism  of $S_{w,g}$   into $ X_{P,\al} ;$
\item[(d)] The following diagram commutes
\begin{equation}\label{itemd}  
\begin{CD}
 S_{w,g}@>{ \pi}>>  X_{P,\al}  \\
@V \Tr VV @Vp_{XV} VV \\
H_{P,\al} @>=>>H_{P,\al}
\end{CD}. 
\end{equation}

\item[(e)]  For a point $v=(s,t,u)\in H_{P,\al}$ the fiber  $\pi^{-1}(v)=S_{w,g}(s,u,t)$ is an orbit of the centralizer $Z_G(g),$ hence, is isomorphic to a rational curve with two distinct punctures.

\item[(f)]  $S:=\pi(S_{w,g})$ is an open dense subset of $ X_{P,\al} .$
 \item [(g]  $S_{w,g}$ is  irreducible; \end{itemize}.\end{Proposition}

\begin{proof} For  the  sake of completeness we provide the idea of the proof. 
We may assume that $g=y_\al$ since the statements of the Proposition remain valid after conjugation.

 (a) follows from \lemref{absirr} 

Proof of (b).  Since the map $\Tr:G^2\to V$ is surjective, 
for every point $(s,t,u)\in H_{w,\al}$ there is a pair 
$(x,y)\in G^2$ such that $\Tr (x,y)=(s,t,u).$  
 Then  $\tr(w(x,y))=\al$, since  $(s,t,u)\in H_{w,\al}.$  Thus there is a matrix $C$ such that 
 $$Cw(x,y)C^{-1}=w(CxC^{-1},CyC^{-1})=g ,$$ $$  (CxC^{-1},CyC^{-1})\in S_{w,g}, \  \mathrm{Tr}(CxC^{-1},CyC^{-1})=(s,t,u).$$
 
Proof of (c).  It is sufficient to check that $\pi$ is injective on $S_{w,g}.$ 
Assume that there are two pairs $(x,y),(x_1,y_1)\in S_{w,g}$ such that $\pi(x_1,y_1)=\pi(x,y).$
It means   there exist a lower triangular matrix $z\in B$ such that $$zxz^{-1}=x_1  , zyz^{-1}=y_1 .$$
 Then $$y_\al=g=w(x,y)=w(x_1,y_1)=zw(x,y)z^{-1}=zy_\al z^{-1}.$$
 Hence, $z\in Z_G(y_\al)\cap B$ which implies that $z=\pm Id.$
Hence, $x=x_1, y=y_1.$ Thus $\pi$ is injective on $S_{w,g}.$

Since $\dim( X_{P,\al})=3,$ $  X_{P,\al} $ is irreducible and 
 $\dim( S_{w,g})\ge 3,$ it means that $\dim S_{w,g}= 3,$  and  $   \pi \mid_{S_{w,g}} $  is   dominant  and birational.

Item (d)   follows from \diagramref{4d} and the fact that $p_{XV}\circ\pi=\Tr.$ 

Proof of (e).  
 We have to show that $S_{w,g}(s,u,t)$ is a rational curve with two punctures.   Since the fiber  $S_{w,g}(s,u,t)$ is mapped injectively into 
$ p_{XV}^{-1}(s,t,u)\cong\BP^1, $ it is a rational curve. 
 $S_{w,g}(s,u,t)$ obviously contains an orbit  $O$  of the action by conjugation of the centralizer
$Z_G(g)$ which is isomorphic to a rational curve with two punctures (see {\bf Notation 3}(5) ).  Let us show that  $S_{w,g}(s,u,t)=O,$  i.e. these two punctures do not belong to $S_{w,g}(s,u,t).$
Assume that $(x,y)\in  S_{w,g}(s,u,t)\setminus O.$  Since there are at most two such points,   $(x,y)$ is invariant under the action of   $Z_G(g).$
But if $zxz^{-1}=x,zyz^{-1}=y$ for some $z\in  Z_G(g)$ then both commute with $g=y_\al$ (see {\bf Notation 3}(5) )  and $\Tr(x,y)\in M,$  which contradicts to the condition 
$v=(s,t,u)\in H_{w,\al}.$

(f)  follows from (c) since  it says that 
  $S=\pi(S_{w,g})$ is an image of an affine set under birational injective morphism into an irreducible variety of the same dimension.

(g)  follows from the fact that $(S_{w,g})$ is birational to an irreducible variety $X_{P,\al}.$

\end{proof}

\begin{Lemma}\label{coniclem:product} Assume that $n,m$ are non-zero integers, $\al\ne \pm 2,$ \ 
$w=[x^n,y^m]\in[\CF_2,  \CF_2],$ and  $P:=P_w$ its trace polynomial.  Then $X_{P,\al}$ is birational to $H_{P,\al}\times \BP^1.$\end{Lemma}
\begin{proof} 
Take   $\be \  | \  \be^2=\al-2$ and define $\gamma\in\BC$ such that  $$\gamma=\frac{\be}{\phi_n(s)\phi_m(t)}.$$ Then $J(s,t,u)=\ga^2$   on $H_{P,\al}.$

 A  birational map  $X_{P,\al}\to \tilde X$ defined  by

 \begin{equation}\label{tau}\tau=\tilde a-\frac{\tilde bt}{2}-\tilde c\frac{s}{2}, \  \varkappa=\frac{\tilde c(st-2u)}{2}+\frac{(t^2-4)\tilde b}{2},\ \rho=\tilde c.\end{equation}

maps $X$ to \begin{equation}\label{tau1} \tilde X:=\{\varkappa^2-(t^2-4)\tau^2=\gamma^2\rho^2\}\subset V\times\BP^3_{\varkappa:\tau:\rho} 
\end{equation}
 and is an isomorphism outside  $\{t^2=4\}.$ The natural projection  of  $\tilde X$ onto $V$ has a section $\{\tau=0, \varkappa=\ga\rho\},$
 Thus $\tilde X$ and $X_{P,\al}$  are   birational to $H_{P,\al}\times \BP^1.$

\end{proof} 

\begin{Proposition}\label{coniclem:nonegative} Assume that $n,m$ are non-zero integers,   $g\in  G, \tr(g)=\al\ne  \pm 2,$ \ 
$w=[x^n,y^m]\in[\CF_2,  \CF_2],$ and  $P:=P_w$ its trace polynomial.  Then logarithmic Kodaira dimension  $\ov{\kappa}(S_{w,g})\ge 0.$ \end{Proposition}

\begin{proof} Let us denote $ S=\pi(S_{w,g})\subset X_{P,\al}.$  In view of \propref{prop:main}(c) it is sufficient to prove  that $\ov{\kappa}(S)\ge 0.$
We know  (thanks to \propref{prop:main}(e) ) that the  fibers of  restriction of $p_{XV}$ on $ S$ are all rational curves with two punctures and 
(thanks to \propref{traceprop:kappa}) that $\ov{\kappa}(H_{P,\al})\ge 0.$

 Now the proposition follows from the inequality  due to O. Fujino, see   \remarkref{kodaira-inequalities}.\end{proof}

\section{Makar-Limanov invariant.}\label{ML}
\

Let $A$ be a ring that is an algebra over $\BC.$ A locally nilpotent derivation $\partial: A\to A$ is a $\BC-$
linear map meeting the following conditions:\begin{itemize}\item $\partial(ab)=\partial(a)b+a\partial(b),$ where $a,b\in A.$
\item for any $a\in A$ there exist  $n\in\BN$ such that $\partial^n(a)=0.$  The kernel $A^{\pl}$ of $\pl$ is a subring of $A.$\end{itemize}
The set of all locally nilpotent derivations on the ring $A$ we denote by $\LND(A).$  The details on locally 
  nilpotent derivations may be found in \cite {Fr}.
 A locally nilpotent derivation $\pl\in\LND(A)$  defines a homomorphism  $\mathrm{exp}:A\to A[\lambda]$ defined by 
$$\mathrm{exp}(a)=1+\lambda\pl(a)+\lambda^2\frac{\pl^2}{2}(a)\dots$$


 Let $X$ be a smooth  affine variety over $\BC$ with the ring of regular functions $A:=\CO(X).$ 

If $\pl\in\LND(A),$   then   for every $\lambda\in\BC$ the operator $$ \mathrm{exp}(\lambda\pl):A\to A, \ f\mapsto
(1+\lambda\pl+\lambda^2\frac{\pl^2}{2}\dots)(f)$$
 is an automorphism of $\CO(X)$ inducing an automorphism of $X.$  Thus 
every $\pl\in\LND(A)$ gives rise to a $\BC^+$ action  $\varphi_{\pl}:\BC^+\times X \to X$ with the following properties
\begin{itemize}\item for every   $  \lambda\in\BC$ the map $\varphi_{\pl}(\lambda) :X\to X, \ (x)\to \varphi_{\pl}(\lambda,x)$ is an automorphism of $X;$
\item for every $x\in X$ the orbit $o(\varphi_{\pl},x)=\{\varphi_{\pl}(\lambda,x),  \lambda\in\BC\}$  is either  a point or is isomorphic to $\BC;$
\item $\varphi_{\pl}(\lambda+\mu)=\varphi_{\pl}(\lambda)\circ
\varphi_{\pl}(\mu);$\item  if $f\in  A$ then $f\in \ker(\pl)=A^{\pl}$ if and only if $f$ is constant on   the orbit  $o(x)$ for every $x\in X.$
\end{itemize}

\begin{Definition}   Let $X$ be an affine variety with the ring of regular functions $A:=\CO(X).$  The Makar-Limanov invariant $\ML(X)$ is defined as 

$$\ML(X)=\bigcap\limits_{\pl\in\LND(A)}A^{\pl}.$$
\end{Definition}

By construction, $\ML(X)$ is a subring of $\CO(X)$  consisting of all the functions on $X$ that are invariant under all $\BC^+-$ actions on $X. $ 
 If 
$ML(X)=\CO(X)$ (which means    $\LND(\CO(X))=\{0\}$) then $X$ admits no  $\BC^+-$ actions, if $ML(X)=\BC$  then the group generated by $\BC^+-$ actions has an open orbit.

\begin{Example}\label{mlinvariant}  Let a surface $S=C_1\times C_2$ be a product of two  smooth affine curves $C_1,C_2.$ Then 
\begin{itemize}\item $\ML(S)=\BC$ if both $C_1$ and $C_2$ are isomorphic to $\BC.$  Indeed , $S\cong\BC^2_{x,y}.$  One can define
$\pl_1=\frac{\partial }{\partial x}$ and  $\pl_2=\frac{\partial }{\partial y}$  and $A^{\pl_1}\cap A^{\pl_2}$ consists of constant functions.
\item $\ML(S)=\CO(C_1)$ if $C_2$ is isomorphic to $\BC_x$  and $C_1$ is not.  $\ML(S)=\CO(S)^{\pl},$ where $\pl=\frac{\partial }{\partial x}.$
The induced $\BC^+-$action  is  $\varphi_{\pl}(\lambda,(x,c))=(x+\lambda,c)$ for a point $(x,c)\in \BC_x\times C_2.$
 Since an orbit of a $\BC^+$-action is isomorphic to $\BC$ it cannot be mapped dominantly to $C_1.$ Thus 
 every other element of $\LND(S)$ induces a $\BC^+-$ action with the same general orbit.
\item $\ML(S)=\CO(S)$  if both $C_1$ and $C_2$ are  not  isomorphic to $\BC.$
In this case $\LND(\CO(S))=\{\pl_0\},$  where $\pl_0(f)=0$ for all $f.$  There are no  curves    isomorphic to $\BC$ in $S.$
\end{itemize}\end{Example}.

 Assume that $X$ is affine and $\ML(X)\ne\CO(X).$  
  Then    $\Aut(X) $ contains a subgroup $H=\{\varphi_{\pl}(\lambda), \lambda\in\BC\}$  isomorphic to 
  $\BC^+.$  For a general point $x$ the orbit $o(H,x) $ is  isomorphic  to $\BC$ and , hence $\ov{\kappa}(o(H,x))=-\infty.$ 
    According to  \eqnref{Iitaka1} (see  \remarkref{kodaira-inequalities})
$$\ov{\kappa}(X)\le\ov{\kappa}(o(H,x))+\dim X-\dim o(H,x)=-\infty,$$
thus \begin{equation}\label{ml:infty}\ov{\kappa}(X)=-\infty.\end{equation}

\begin{Theorem}\label{mlcor:trivial}  Assume that $n,m$ are  non-zero integers,   $g\in\SL(2,\BC) , \tr(g)=\al\ne  \pm 2,$ \ 
$w=[x^n,y^m]\in[\CF_2,  \CF_2].$  
 Then $$\ML(S_{w,g})=\CO(S_{w,g}).$$\end{Theorem}

\begin{proof}  Indeed,  assume that  $\ML(S_{w,g})\ne\CO(S_{w,g}).$ Then  $S_{w,g}$ admits a $\BC^+-$action and $\ov \kappa(S_{w,g})=-\infty.$  But according to \propref{coniclem:nonegative} $\ov \kappa(S_{w,g})\ge 0.$ 
The contradiction shows that $\ML(S_{w,g})=\CO(S_{w,g}).$\end{proof}


\section{Jordan properties  of $\Aut(S_{w,g})$.}\label{jordan}

 We are working with quasiprojective threefolds and use the following   facts.
 \begin{Remark}\label{jordanprop:known}
   \begin{enumerate}
   \item  The group $GL(n,\BC)$ is Jordan   (\cite{Jordan} );
 \item The Cremona group $  \mathrm{Cr}_n$ of birational selfmaps of $\BP^n$ is Jordan (\cite [Theorem 5.3]{Serre1}, \cite[Theorem 3.1]{Serre2}\cite{PS16});
\item   Assume that $W$ is a quasiprojective    irreducible variety of dimension $d\le 3$  that  is not  birational to $E\times \BP^2, $  where $E$ is an elliptic curve.  Then $\Aut (W)$ is Jordan (\cite[Corollary 5]{BZopen}).

 \end{enumerate} \end{Remark}

\begin{Theorem}\label{jordanthm:jordan} 
 Let  $n,m$ be non-zero integers.  Let  $\al\ne\pm 2$ be a complex number, and $g\in \SL(2,\BC) $ a  matrix with  trace   $\tr(g)=\al.$  Let $S_{w,g}:=\{(x,y)\in \SL(2,\BC)\times \SL(2,\BC) \ | w(x,y)=g\}.$Then  \begin{itemize}\item[(a)]$\Aut(S_{w,g})$ is Jordan;
%
\item[(b)]  If $ w\in \CF_2 , \ w=[x^n,y^m]$  and $|n|>2$ and  $|m|>2$ then $\Aut(S_{w,g})$  is very Jordan.\end{itemize}\end{Theorem}

\begin{proof}   According to \propref{prop:main},   it is sufficient to prove that   the group $\Aut(S)$ is Jordan (resp., very Jordan), where $S=\pi(S_{w,g})\subset 
X_{P_w,\al}:=X_{\al}.$  By \lemref{coniclem:product}
$X_{P_w,\al}$   is birational to $H\times \BP^1,$    where $H:=H_{w,\al}.$

 Let $A_0\subset \Aut(S)$ be the subgroup  of such $f\in \Aut(S)$   that  $p_{XV}\circ f$ is constant along the fibers of $p_{XV}.$ 
  If $f\in  A_0,$ it may  be included into  the following commutative diagram
\begin{equation}\label{cd100}  
\begin{CD}
S @>{ f}>> S  \\
@V p_{XV} VV @Vp_{XV} VV \\
H @>{f_*}>>H
\end{CD}. 
\end{equation}

Here $f_*$ is a rational map (\cite [Lemma 10.7(Kawamata)]{Itaka}, defined everywhere, thus regular.  Moreover, $f_*\in\Aut(H)$ since $f\in\Aut(S),$
thus images under $f$ of different fibers   of  $p_{XV}$ are disjoint.

\begin{Lemma}\label{kappa2} Assume that $\ov \kappa(H)=2.$ Then $A_0$ is  very Jordan. \end{Lemma}
\begin{proof}
 Since $X_{\al}\setminus S$  meets every fiber   of  $p_{XV}$ at two distinct points, 
 we may  use \cite[Lemma  6]{BZopen} (where  $X_{\al}$  and $H$ play roles of $U$ and $V,$ respectively) and extend 
 $f$ to an  automorphism $\tilde f$  of $X_\al.$
  
 The group $A_0$ contains a subgroup $A_1$ of those $f\in  A_0$  for which ${f_*}=id . $  The group $A_1$ contains the subgroup $A_2$  of index at most 2 of those   $f$ for which $\tilde f$ 
 fix all points of $D:=X_\al\setminus S.$

Let us show that $A_2$ is commutative. Let $\CK=\BC(H)$ 
be the field of rational functions on $ H.$ The divisor $D$ is either presented by two points $\mu_1,\mu_2$ in $\BP^1(\CK)$ ( if it is union of two sections) or  a point $\mu$  in 
$\BP^1(\CE),$  where $\CE=\CK[j]$ is a quadratic extension of $\CK, j^2\in\CK.$  
  
Group  $A_2$ is a subgroup  of a group  $B\subset\PGL(2,\CK) $   that consist of those elements  of  $\PGL(2,\CK) $  that fix    $\mu_1$ and $\mu_2,$ if there are two sections or 
 $\mu$ otherwise. 
 The group $B$ is either a split algebraic torus (if $\mu_i$ are defined over $\CK$) or non-split algebraic torus. In any case it is an abelian group. 
 Thus its subsgroup $A_2$ is abelian as well, and $A_1$ is very Jordan,  Since$[A_1:A_2]\le 2.$
  
   By construction there is a short exact sequence:

$$0\to A_1\to  A_0\to\Aut(H)$$

Since $\ov{\kappa}(H)=2, $  the group $\Aut(H)$ is finite ( \cite{Sakai}, see \remarkref{kodaira-inequalities}) . 

 It follows that $A_0$ is very Jordan.
\end{proof}

Proof of (a).  We may assume that $|n|\ge|m|\ge  1.$ 
 If $m|=1$ 
 by \corref{tracecor:rational} the surface $H$  is rational. Thus $X_{\al}$  is a rational threefold,
 thanks to \lemref{coniclem:product}.
 Hence $\Aut(S)$ is a subgroup of the Cremona group $\mathrm{Cr}_3,$ that is Jotdan.  It follows that $\Aut(S)$ is Jordan as well.

Assume that $|m|\ge 2.$

In light of item (4) of \remarkref{jordanprop:known} it is sufficient to check 
only case when there is a birational map $\varphi: X_\al\to E\times\BP^2 ,$  where $E$ is an elliptic curve.  
 We have the following commutative diagram 

\begin{Diagram}\label{ep2}
 \begin{alignat}{5}
& &X_\al   &\quad\stackrel{\varphi}{\rightarrow}\quad  &E\times\BP^2&\notag \\
&{p_{XV}} &\downarrow   &  \quad \searrow  {\tilde h}   &  \downarrow  {p_E}   &  \notag \\
& &H          &  \quad\stackrel{h}{\rightarrow}\quad  & E\quad&\end{alignat}
\end{Diagram}

 Here \begin{itemize}\item  $p_E : E\times\BP^2\to E$ is the projection to the first factor;\item  ${\tilde h:=p_E\circ\varphi};$
\item $ h$ is defined, since fibers of $p_{XV}$ are rational curves, thus are mapped  by $p_E\circ\varphi$   into a point in $E;$
\item $h$ is rational   (\cite [Lemma 10.7(Kawamata)]{Itaka};
\item $h$ is regular, since $E$ is an elliptic curve;
\item  for a general  point $e\in E$ the fiber $h^{-1}(e)=p_{XV}(\varphi^{-1}p_E^{-1}( e))\subset H$ is a rational curve because $p_E^{-1}( e)\cong \BP^2.$
\end{itemize}
  It follows that $H$ is birational to $E\times \BP^1.$  Thus, by \lemref{elliptic} there are two options.
  
  {\bf 1.} $h$ is not surjective.  Let $f\in\Aut(S). $ Then since $\BP^2$ cannot be mapped dominantly into $E$, automorphism  $ \tilde f $ may be included into the following commutative diagram
  
  \begin{Diagram}\label{ep3}
 \begin{alignat}{5}
&   X_\al &\stackrel{\tilde f}{\rightarrow}  & X_\al&\notag \\
&\downarrow  {\tilde h}   &     &  \downarrow  {\tilde h}   &  \notag \\
& E          & \stackrel{\tilde f_!}{\rightarrow} & E&\end{alignat}
\end{Diagram}
  
 where  ${\tilde f_!}$    is an automorphism of $E':=h(S)\subset E.$  Let $A_E\subset \Aut(S)$ be the subgroup of those $f\in\Aut(S)$ for which $\tilde f_!= id.$   The group $A_E$ is a subgroup of the Cremona group $	\mathrm{Cr}_2(K) ,$  where $K=\BC(E),$ hence is Jordan (see \remarkref{jordanprop:known}.)
  We have the following exact sequence:
 
$0\to A_E\to\Aut(S)\to\Aut(E').$ 
 
 Since $\Aut(E')$  is  finite and  $A_E$ is Jordan, $\Aut(S)$ is Jordan as well.

   {\bf 2.} The general fiber $h^{-1}(e)$ is a rational curve with at least  3 punctures.   In this case $\ov{\kappa}|(H)=2$ and   $\Aut(S)=A_0.$  
Indeed,  assume that $f\in \Aut(S)$ and $v=(s,t,u)$ is a point of $H.$ The fiber $ F_v:=p_{XV}^{-1}(v)\cap S\subset  S$ is a conic - a rational curve with two punctures.   
Hence  $ F_v:=p_{XV}^{-1}(v)$ cannot be mapped dominantly to $E$ and cannot be mapped dominantly to a fiber of $h$ since the last one has at least three punctures. Thus $p_{XV}(f(F))$ has to be a point in $H.$
     Thanks to \lemref{kappa2} the group $Aut(S)=A_0$ is very Jordan, hence Jordan.
 
 Proof of (b).  In this case  $s$ and $t$ omit at least three values of $\BP^1$ (roots of $\phi_n(s)$ and $\phi_m(t),$ respectively, and $\infty$) on $H.$
 Thus every curve along which $s$ or $t$ is not constant has to have at least three punctures. 
 
  For a point $v\in H$ the fiber   $F_v:=p_{XV}^{-1}(v)\cap S\subset  S$  has only two punctures, and the same should be valid for a curve $C_v:=p_{XV}(f(F_v))$ for any automorphism $f\in \Aut(S).$    It  follows that $s$ and $t$ are constant along $C_v$ and $C_v$ is also a fiber of $p_{XV},$
 i.e., $f\in A_0$   and   $Aut(S)=A_0.$   Now (b) follows from \propref{traceprop:kappa}  and \lemref{kappa2}.
\end{proof}

\section{Decalarations}

{\ }

{\bf Data availability} No datasets were generated or analized during the current study

{\ }

{\bf Conflict of interest} The author  has not disclosed any competing interests.


\begin{thebibliography}{99999}



\bibitem[Bo]{Bo}
A. Borel, {\it On free subgroups of semisimple groups}, Enseign.
Math. {\bf 29} (1983) 151--164; reproduced in {\OE}uvres - Collected
Papers, vol.~IV, Springer-Verlag, Berlin--Heidelberg, 2001,
pp.~41--54.


\bibitem[BG]{BG} T. Bandman, S. Garion,  {\em
 Surjectivity and equidistribution of the word $x^ay^b$ on $PSL(2,q)$ and $SL(2,q),$}
International Journal of Algebra and Computation (IJAC),{\bf 22}(2012), n.2, 1250017--1250050.
\bibitem[BGK]{BGK}
T.~Bandman, S. Garion, B.~Kunyavski\u\i , {\it Equations in simple
matrix groups: algebra, geometry, arithmetic, dynamics}, Central
European J. Math. {\bf 12} (2014), 175--211.





\bibitem[BK]{BK}
T. Bandman, B. Kunyavski\u\i , {\it Criteria for equidistribution of
solutions of word equations on $\SL(2)$}, J. Algebra {\bf 382}
(2013), 282--302.


\bibitem[BKS]{BKS}  T. Bandman, B. Kunyavskii, A.N. Skorobogatov
{\it Birational properties of word varieties},
arXiv:2504.15461  


 \bibitem[BZ15]{BandmanZarhinTG} T. Bandman, Yu.G. Zarhin, {\sl Jordan groups and algebraic surfaces}. 	Transformation Groups {\bf 20} (2015), no. 2, 327--334. 
\bibitem[BZ16]{BZ}
T. Bandman, Yu. G. Zarhin,
{\it Surjectivity of certain word maps on $PSL(2,\mathbb C)$ and $SL(2,\mathbb C)$},
Eur. J. Math. {\bf 2} (2016), 614--643.
 \bibitem[BZ20]{BZ20}  T. Bandman, Yu.G. Zarhin, {\sl  Bimeromorphic automorphism groups of certain  $\BP^1-$bundles}.  European Journal of Mathematics, {\bf 7}(2021) 641-670.

\bibitem[BZ18]{BZopen} T. Bandman, Yu.G. Zarhin, {\sl  Jordan properties of automorphism  groups of certain open algebraic varieties}. Transform. Groups  {\bf  24} (2019), no. 3, 721–739. 
\bibitem[BZ22]{BZ22} T. Bandman, Yu.G. Zarhin,  {\sl Automorphism groups of $\BP^1$-
bundles over a non-uniruled base}.  Russian Mathematical Surveys  {\bf 78},(2023) n.1, 1--64. 

\bibitem[BZ24]{BZ24}T. Bandman, Yu. G. Zarhin,
{\it Jordan groups and geometric properties of manifolds}
Arnold Math. J. 10 (2024), no. 4, 621–635.




\bibitem[Fre]{Fr}   G.Freudenburg.{\it
Algebraic Theory of Locally Nilpotent Derivations},
Encyclopaedia of Mathematical Sciences
Springer-Verlag, 2006.

\bibitem[Fri83]{Fri83} S. Friedland. Simultaneous similarity of matrices.
{\em Adv.~Math.} {\bf 50} (1983) 189--265.

\bibitem[Fu]{Fuj}  O.  Fujino, 
{\it Iitaka conjecture, an Introduction}, 
SpringerBriefs in Mathematics,Springer Singapore,   2020.





\bibitem[Gol]{Gol}
W. M. Goldman,
{\it Trace coordinates on Fricke spaces of some simple hyperbolic surfaces},
in: ``Handbook of Teichm\"uller Theory. Vol.~II'' (A.~Papadopoulos, Ed.),
IRMA Lect. Math. Theor. Phys., vol.~13, Eur. Math. Soc., Z\"urich, 2009,
pp.~611--684.
\bibitem[Ho]{Ho}
R. D. Horowitz, {\it Characters of free groups represented in the
two-dimensional special linear group}, Comm. Pure Appl. Math. {\bf
25} (1972), 635--649.
\bibitem [Iit77] {Iit77}  S. Iitaka,
{\it On logarithmic Kodaira dimension of algebraic varieties},  Complex analysis and algebraic geometry, pp. 175–189
Iwanami Shoten Publishers, Tokyo, 1977

\bibitem [Iit78] {Iit78}  S. Iitaka,
{\it Some applications of logarithmic Kodaira dimension},  Proceedings of the International Symposium on Algebraic Geometry (Kyoto Univ., Kyoto, 1977), pp. 185–206
Kinokuniya Book Store Co., Ltd., Tokyo, 1978


\bibitem [Iit]
{Itaka} Sh.  Iitaka, Algebraic Geometry, Graduate Texts in Mathematics, Vol.  76,  Springer-Verlag, Berlin Heidelberg New York, 1982
\bibitem[Jor]{Jordan} C. Jordan,  {\it  M\'{e}moire  sur des \'{e}quations differentielles  lin\'{e}ares   \`{a} int\'{e}grale alg\'{e}brique}.  Crelle's
journal  {\bf  84}, (1878)   89-215: Ouvres II, 13-140.
\bibitem[KKMLR]{KKML} S. Kaliman, M. Koras, L. Makar-Limanov, P.  Russell, {\it  $\BC^*$ -actions on  $\BC^3$  are linearizable},
Electron. Res. Announc. Amer. Math. Soc. 3 (1997), 63–71.


\bibitem[LMZ]{LMZ}  Jujie Luo, Sheng Meng, De-Qi Zhang,{\it Jordan Property for automorphism groups
of compact varieties}, EMS math.sci Surveys, 2025, online.

\bibitem[LOST]{LOST}
M. W. Liebeck, E. A. O${}^{\prime}$Brien, A. Shalev, P. H. Tiep,
{\it The Ore conjecture}, J. Europ. Math. Soc. {\bf 12} (2010),
939--1008.
\bibitem[ML96]{ML96}   L. Makar-Limanov, 
{\it On the hypersurface $x+x^2y+z^2+t^3=0$ in $\BC^4$  or a $\BC^3$-like threefold which is not C3},
Israel J. Math. 96 (1996), 419–429.

\bibitem[ML01]{ML}
L. Makar-Limanov, 
{\it AK-invariant, some conjectures, examples and counterexamples},
Ann. Polon. Math. 76 (2001), no. 1-2, 139–145.

\bibitem[MP]{MP} G. L. Mullen, D. Panario {Handbook  of Finite fields,}    Discrete Mathematics and its Applications  series, CRC Press,  2013.





\bibitem[MZ]{MengZhang} Sh. Meng, D.-Q. Zhang, {\sl Jordan property for non-linear algebraic groups and projective varieties}.  Amer. J. Math.  {\bf 140:4} (2018), 1133--1145.

\bibitem[Po11]{Pop} V.L.  Popov,  {\it On the Makar-Limanov, Derksen invariants, and finite automorphism groups
of algebraic varieties}. In: Affine algebraic geometry, 289--311, CRM Proc. Lecture Notes {\bf 54}, Amer. Math. Soc., Providence, RI, 2011.

\bibitem[PS14]{PS14}  Yu. Prokhorov, C. Shramov, {\it Jordan property for groups of birational selfmaps}. Compositio
Math. {\bf 150} (2014),  2054--2072.




\bibitem[PS16]{PS16} Yu. Prokhorov, C. Shramov,  
{\it Jordan property for Cremona groups},
Amer. J. Math. 138 (2016), no. 2, 403–418.


\bibitem[Sa]{Sakai}  F. Sakai, 
{\it Kodaira dimensions of complements of divisors},Complex analysis and algebraic geometry, pp. 239–257
Iwanami Shoten Publishers, Tokyo, 1977

\bibitem[Se09a]{Serre1} J-P. Serre, {\sl A Minkowski-style bound for the orders of the finite subgroups of
the Cremona group of rank 2 over an arbitrary field}. Moscow Math. J. {\bf 9}
(2009), no. 1, 183--198.

\bibitem[Se09b]{Serre2}{\it
Le groupe de Cremona et ses sous-groupes ﬁnis},
Astérisque, tome 332 (2010), S\'eminaire Bourbaki,  Nov. 2008,   exp. no 1000, p. 75-100



\bibitem[Thp]{Thp}
R. C. Thompson,
{\it Commutators in the special and general linear groups},
Trans. Amer. Math. Soc. {\bf 101} (1961), 16--33.

\bibitem[Zar14]{Zar14} Yu.G. Zarhin, {\sl Theta groups and products of abelian and rational varieties}. Proc. Edinburgh Math. Soc. {\bf 57:1} (2014), 299--304.

\end{thebibliography}
\end{document}